\documentclass{amsart}
\usepackage[T1]{fontenc}
\usepackage[english]{babel}
\usepackage{amsmath,amsthm,amssymb, latexsym,fancyhdr,stackrel,geometry,enumerate,bbm}
\usepackage[all]{xy}
\usepackage{srcltx}
\usepackage[left,modulo, pagewise]{lineno}

\usepackage{graphicx}
\usepackage{caption}
\usepackage{subcaption}

\usepackage{hyperref}
\usepackage{srcltx}

\date{\today}
\swapnumbers
\theoremstyle{plain}
\newtheorem{thm}{Theorem}[section]
\newtheorem{lem}[thm]{Lemma}
\newtheorem{prop}[thm]{Proposition}
\newtheorem{cor}[thm]{Corollary}
\newtheorem*{Thm}{Theorem}
\newtheorem{conj}[thm]{Conjecture}

\theoremstyle{definition}

\newtheorem{rem}[thm]{Remark}

\newcommand{\ts}[1]{\normalfont{\textsf{#1}}}
\newcommand{\K}{\Bbbk}

\newcommand{\ol}[1]{\overline{#1}}

\renewcommand{\a}{\alpha}
\renewcommand{\b}{\beta}
\renewcommand{\c}{\gamma}
\newcommand{\e}{\varepsilon}
\renewcommand{\l}{\lambda}
\renewcommand{\d}{ \delta}

\newcommand{\f}{\varphi}
\newcommand{\td}{\tilde{\d}}

\renewcommand{\r}{\mathfrak{r}}
\renewcommand{\t}[1]{\widetilde{#1}}

\newcommand{\cal}[1]{\mathcal{#1}}

\newcommand{\R}{\mathcal{R}}
\newcommand{\C}{\mathcal{C}}

\newcommand{\mor}[3]{$#1\colon #2 \to #3$}

\renewcommand{\ker}[1]{\ts{Ker}\,  #1}
\newcommand{\im}[1]{\ts{Im}\,  #1}

\newcommand{\Hom}[3]{\ts{Hom}_{#1}\!(#2,#3)}
\newcommand{\Ext}[4]{\ts{Ext}_{#1}^{#2}\!(#3,#4)}
\newcommand{\End}[2]{\ts{End}_{#1}\!(#2)}

\newcommand{\Ho}[3]{\ts{HH}^{#1}(#2,#3)}
\newcommand{\HH}[2]{\ts{HH}^{#1}(#2)}

 \title[$\ts{HH}^1$  for cluster-tilted algebras  ]{The first Hochschild cohomology  group of  a cluster-tilted algebra revisited}

\author[I. Assem]{Ibrahim Assem}
\address{D\'epartement de Math\'ematiques, Universit\'e
  de Sherbrooke, Sherbrooke, Qu\'ebec, Canada J1K~2R1}
\email{Ibrahim.Assem@usherbrooke.ca}

\author[J.C. Bustamante]{Juan Carlos Bustamante}
\address{D\'epartement de Math\'ematiques, Universit\'e
  de Sherbrooke, Sherbrooke, Qu\'ebec, Canada J1K~2R1}
\email{juan.carlos.bustamante@usherbrooke.ca}

\author[K. Igusa]{Kiyoshi Igusa}
\address{Department of Mathematics, Brandeis University, Waltham, MA 02454}
\email{igusa@brandeis.edu}

\author[R. Schiffler]{Ralf Schiffler}
\address{Department of Mathematics, University of Connecticut, Storrs, CT 06269}
\email{schiffler@math.uconn.edu}

 
\begin{document}
\maketitle 

\begin{center}
 Dedicated to Christophe Reutenauer for his $60^{\rm th}$ birthday.
\end{center}

\begin{abstract}
	Given a cluster-tilted algebra $B$ we study its first Hochschild cohomology group $\HH{1}{B}$ with coefficients in the $B\textrm{-}B$-bimodule $B$. If $C$ is a tilted algebra such that $B$ is the relation extension of $C$ by $E=\Ext{C}{2}{DC}{C}$, then we prove that $\HH{1}{B}$ is isomorphic, as a vector space, to the direct sum of $\HH{1}{C}$ with $\Ho{1}{B}{E}$ . This yields homological interpretations for results of the first and the fourths author with M.J. Redondo. 
\end{abstract}

\section*{Introduction}\label{section:intro}
This paper is the third of a series devoted to studying the first Hochschild cohomology group of a cluster-tilted algebra \cite{AR09, ARS12}.

Cluster-tilted algebras appeared naturally during the study of the cluster algebras of Fomin and Zelevinsky \cite{FZ02}. They were introduced in \cite{BMR06} and independently in \cite{CSS06} for the type $\mathbb{A}$ and, since then, have been the subject of several investigations. In particular, it was proved in \cite{ABS08} that if $C$ is a tilted algebra, then the trivial extension of $C$ by the $C\textrm{-}C$-bimodule $E=\Ext{C}{2}{DC}{C}$, called the \emph{relation-extension} of $C$, is cluster-tilted and, conversely, every cluster-tilted algebra arises in this way.

The Hochschild cohomology groups of an algebra were defined by Hochschild in 1945, see \cite{H45}. These are subtle homological invariants, not only of the algebra, but also of its derived category \cite{Hap89, K04}. In \cite{Skow92}, the vanishing of the first Hochschild cohomology group was related to the simple connectedness of the algebra. Further connections between the first Hochschild cohomology group and the fundamental groups of an algebra were obtained in \cite{AdlP96,PS01}. It was then a natural question to try to relate the first Hochschild cohomology group $\HH{1}{C}$ of a tilted algebra $C$ with coefficients in the $C\textrm{-}C$-bimodule $C$ to the corresponding group $\HH{1}{B}$ of the relation-extension $B$. For this purpose, a first observation is that, because $B$ is a trivial extension of $C$, then there exists a canonical morphism \mor{\f}{\HH{1}{B}}{\HH{1}{C}}, see \cite{AR09}. Next, an equivalence relation was defined in \cite{AR09} between the arrows in the quiver of $B$ which are not in the quiver of $C$. The number of equivalence classes is then denoted by $n_{B,C}$. It was shown in \cite{AR09} that if $C$ is a tilted algebra over an algebraically closed field $\K$ such that the relation-extension $B$ is schurian, then there exists a short exact sequence of vector spaces
$$\SelectTips{eu}{10}\xymatrix{ 0\ar[r] & \K^{n_{B,C}} \ar[r] &\HH{1}{B} \ar[r]^\varphi & \HH{1}{C}\ar[r] & 0 . }$$
This result was generalised in \cite{ARS12} to the cases where $C$ is constricted (in the sense of \cite{BM01}) or $B$ is tame. The proofs of these two results were combinatorial. In the case of representation-finite cluster tilted-algebras, the Hochschild cohomology has also been computed by Ladkani using different methods  see \cite{L12-arXiv}.

Our objective in this note is to provide a homological interpretation of this short exact sequence, removing all assumptions on $B$ or $C$. Our main theorem may be stated as follows.

\begin{Thm} Let $\K$ be an algebraically closed field , $C$ be a tilted $\K$-algebra and $B$ be the trivial extension of $C$ by $E=\Ext{C}{2}{DC}{C}$. Then there exists a short exact sequence of vector spaces
$$\SelectTips{eu}{10}\xymatrix{ 0\ar[r] & \Ho{1}{B}{E} \ar[r] &\HH{1}{B} \ar[r]^\f & \HH{1}{C}\ar[r] & 0 . }$$
\end{Thm}
The proof of this theorem is largely homological and  different from those in \cite{AR09, ARS12}. We also prove that $\ts{dim}_\K \Ho{1}{B}{E} \geqslant n_{B,C}$ and equality holds if and only if the indecomposable summands of the $C\textrm{-}C$-bimodule $E$ are orthogonal bricks. Combining this with the results of [6,5] we obtain, under the hypotheses therein, that, as a $C-C$-bimodule, $E$ is a direct sum of exactly $n_{B,C}$ orthogonal bricks.

Our paper is organised as follows. After a short preliminary section 1, we start the proof of our theorem in section 2 by proving the left exactness of the required sequence. It is next shown in section 3 to be right exact and  we study the kernel of the map $\f$ in section 4. We end the paper with an example in section 5.

\section{Preliminaries}\label{section:prelim}

\subsection{Quivers and relations} While we briefly recall some  concepts concerning bound quivers and algebras, we refer the reader to \cite{ASS06} or \cite{ARS95}, for instance, for unexplained notions.

Let $\K$   be a commutative field. A \emph{quiver} $Q$ is the data of two sets, $Q_0$ (the \emph{vertices}) and $Q_1$ (the \emph{arrows}) and two maps \mor{s,t}{Q_1}{Q_0} that assign to each arrow $\a$ its \emph{source} $s(\a)$ and its \emph{target} $t(\a)$. We write \mor{\a}{s(\a)}{t(\a)}. If $\b\in Q_1$ is such that $t(\a)=s(\b)$ then the composition of $\a$ and $\b$ is the \emph{path} $\a\b$. This extends naturally to paths of arbitrary positive length. The \emph{path algebra} $\K Q$ is the $\K$-algebra whose basis is the set of all paths in $Q$, including one stationary path $e_x$ at each vertex $x\in Q_0$, endowed with the  multiplication induced from the composition of paths. If $|Q_0|$ is finite, the sum of the stationary paths is the identity.

In case $\K$ is algebraically closed, then any finite-dimensional basic and connected algebra $A$  can be obtained as a quotient of a path algebra  $A\simeq \K Q/I$. In this case, the pair $(Q,I)$ is called a \emph{bound quiver}. Given two vertices $x,y\in Q_0$, a \emph{relation} from $x$ to $y$ is a $\K$ - linear combination $r= \sum_{i=1}^m \l_i w_i \in e_x I e_y$ of paths $w_i$ of length at least two from $x$ to $y$. The relation $r$ is \emph{minimal} if none of the scalars $\l_i$ is zero, and for any proper subset $J\subset \{1,\ldots,m\}$ one has $\sum_{j \in J} \l_j w_j \not\in e_x I e_y$. The relation $r$ is said to be \emph{strongly minimal} if, as before, $\l_j\not=0$ and for any proper subset $J\subset \{1,\ldots,m\}$ there is no family of non-zero scalars $\mu_j$ such that $\sum_{j\in J} \mu_j w_j \in e_x I e_y$. 

Given an algebra $A \simeq \K Q/I$, a \emph{system of relations} for an algebra $A$ is a subset $\mathcal{R}$ of $\bigcup_{x,y\in Q_0} e_x I e_y$ that generates $I$ as a two-sided ideal, but such that no proper subset of $\mathcal{R}$ does. It is shown in \cite[2.2]{ARS12} that one may assume $\R$ to be a system of strongly minimal relations.

\subsection{Cluster-tilted algebras}\label{subsec:cluster-tilted} Let $H$ be a finite dimensional hereditary $\K$-algebra, $\ts{mod}\textrm{-}H$ the category of finite dimensional right $H$-modules and  $\cal{D}^b(\ts{mod}\textrm{-}H)$ the corresponding bounded derived category. It is a triangulated category with shift functor denoted by $[1]$, and it has an Auslander-Reiten translation $\tau$. The \emph{cluster category} of $H$ is the orbit category $\cal{C}_H : = \cal{D}^b(\ts{mod}\textrm{-}H)/\tau^{-1}\circ [1]$. Again, it is a triangulated category having almost split triangles. An object $T$ in $\cal{C}_H$ is a (basic) \emph{tilting} object if $\Ext{\cal{C}_H}{1}{T}{T}=0$ and $T$ is the sum of $\ts{rk}\  \ts{K}_0(H)$ indecomposable objects which in addition are pairwise  non-isomorphic. The endomorphism algebra $\End{\cal{C}_H}{T}$ is a \emph{cluster-tilted algebra}.

\medskip 

Consider a tilting module $U$ over a hereditary algebra $H$, so that the algebra $C=\End{H}{U}$ is a  \emph{tilted algebra} \cite{HR82} and denote by $D$ the standard duality $\Hom{\K}{-}{\K}$ between  $\ts{mod}\textrm{-}H$ and $\ts{mod}\textrm{-}H^{op}$. Let $E$ be the $C\textrm{-}C$-bimodule $E=\Ext{C}{2}{DC}{C}$ with the natural actions. The trivial extension $C\ltimes E$ of $C$ by $E$, called the \emph{relation-extension} of $C$, is the algebra whose underlying vector space is $C \oplus E$, endowed with the multiplication induced by the bimodule structure of $E$, namely  $$(c_1,e_1)\cdot(c_2,e_2) = (c_1 c_2, c_1 e_2 + e_1 c_2).$$ It was shown in \cite{ABS08} that $B= C\ltimes E$ is a cluster-tilted algebra, and, conversely, every cluster-tilted algebra arises in this way, though not uniquely.

Also, the natural projection \mor{p}{B}{C} is a morphism of algebras, and so is its right inverse \mor{q}{C}{B}. We have a short exact sequence of $B\textrm{-}B$-bimodules 
\begin{equation}\label{eqn:short-BB}  0\longrightarrow E \stackrel{i}{\longrightarrow} B \stackrel{p}{\longrightarrow} C \longrightarrow 0.
\end{equation}

\medskip Also, it was shown in \cite{ABS08} that once the bound quiver $(Q,I)$ of $C$ is known then that of $B$, say $(\tilde{Q},\tilde{I})$, is obtained as follows:
\begin{itemize}
	\item $\tilde{Q}_0 = Q_0$;
	\item For $x, y \in Q_0$ , the set of arrows in $\t{Q}$ from $x$ to $y$ equals the set of arrows in $Q$ from $x$ to $y$ (which we call \emph{old arrows}) plus $| \R \cap  e_y I e_x|$ additional arrows (which we call \emph{new arrows}).
\end{itemize}
The relations defining $\t{I}$ are given by the partial cyclic derivatives of the potential $W = \sum_{r\in \R} \c_r r$, where $\c_r$ is the new arrow associated to the relation $r$ (see \cite{K11}).  Potentials are considered up to cyclic  permutations : two potentials are \emph{cyclically equivalent} if their difference lies in the linear span of all elements of the form $\a_1 \a_ 2 \cdots \a_j - \a_j \a_1 \cdots \a_{j-1}$, where    $\a_1 \a_ 2 \cdots \a_j  $ is an oriented cycle. We recall from \cite{DWZ10} that, for a given arrow $\b$ the \emph{cyclic partial derivative} $\partial_\b$ of  $W$ is defined on each cyclic summand $\b_1\cdots \b_s$ by $\partial_\b (\b_1\cdots \b_s) = \sum_{i:\b = \b_i} \b_{i+1} \cdots \b_s \b_1 \cdots \b_{i-1}$. In particular, the cyclic derivative is invariant under cyclic permutation.

\subsection{Hochschild cohomology} We recall some notions concerning Hochschild cohomology, but for unexplained ones, we refer to \cite{Hap89, Redondo01} for instance. Given a $\K$-algebra $A$, let $A^e = A \underset{\K}{\otimes}A^{op}$ be its enveloping algebra. It is well-known that the category of $A\textrm{-}A$-bimodules is equivalent to that of $A^e$-modules. If $_A X_A$ is a bimodule, then the Hochschild cohomology groups of $A$ with coefficients in $X$ are the extension groups $\Ho{i}{A}{X} = \Ext{A^e}{i}{A}{X}$. In case $X=A$, we simply write $\HH{i}{A}$.

We are  interested in computing the Hochschild cohomology groups of cluster-tilted algebras, which are given by quivers and relations. In this context we can use a convenient resolution for computing the $\ts{Ext}$ groups, see \cite[1.1 and 1.2]{C89}. Let $A=\K Q /I$ and  $\r$ be its Jacobson radical. Then $A_0=  A / \r$ is the semisimple algebra generated by the vertices of $Q$, and as $A_0$-bimodules one has $A = A_0 \oplus \r.$ The following result will be used in the sequel.

\setcounter{thm}{\value{subsection}}

\begin{prop}[1.2 in \cite{C89}]\label{prop:res-rad} Given an $A^e$-module $X$, the Hochschild cohomology groups $\Ho{i}{A}{X}$ are the cohomology groups of the complex
$$\SelectTips{eu}{10}\xymatrix{0\ar[r] & X^{A_0} \ar[r]^-{d^1} & \Hom{A_0^e}{\r}{X} \ar[r]^{d^2} & \Hom{A_0^e}{\r \otimes \r }{X}\ar[r]^{d^3}  &\Hom{A_0^e}{\r^{\otimes^3 }}{X} \ar[r] &\cdots }$$
 where the tensor products are taken over $A_0$, $X^{A_0} = \{x\in X| sx = xs,\  \mbox{ for all } s \in A_0\} = \underset{s\in Q_0}{\bigoplus} e_s X e_s $, the differentials are given by $(d^1 x)(r) = xr-rx$, and, in general for $i\geqslant 2$
\begin{eqnarray*}
 d^i f (r_1\otimes \cdots \otimes r_i) & = & r_1 f(r_2\otimes \cdots \otimes r_i) + \sum_{j=1}^{i-1}(-1)^j f(r_1\otimes \cdots \otimes r_j r_{j+1} \otimes \cdots r_i)  \\
							& + & (-1)^i f(r_1\otimes \cdots \otimes r_{i-1})r_i.
\end{eqnarray*}
\end{prop}\qed

From this, one sees that $\Ho{0}{A}{X} = \{ x\in X|\ ax = xa, \mbox{ for all } a\in A\}$, so in particular $\HH{0}{A}$ is the centre of $A$. The kernel of the map $d^2$ is the set of $A_0$-bilinear maps \mor{f}{\r}{X} such that $f(a_1 a_2) = a_1 f(a_2) + f(a_1) a_2$ for $a_1,a_2\in \r$, that is, the \emph{derivations} of $\r$ in $X$. If we extend such a derivation $f$ to $A_0$ by letting $f(A_0)=0$, we obtain the derivations of $A$ in $X$ (see \cite{C98,Redondo01}). Also, for a fixed $x\in X $ the map $  d^1 x = [x,-] \colon a\mapsto ax -xa $ is a derivation, and  $\ts{Im}\ d^1$ is the set of \emph{inner} derivations.

A useful feature of the complex above is that we only need to deal with maps that are $A_0$-bilinear. Thus, if $r\in e_i \r e_j$, then  $f(r) = f(e_i r e_j) = e_i f(r) e_j \in e_i X e_j.$

\begin{rem}\label{rem:alt-approach}
Alternatively, derivations can be described as follows. Let $\K Q_1$ be the $A_0$-bimodule generated by the set of arrows of $Q$. A $\K Q_0$-bilinear map \mor{\d}{\K Q_1}{\K Q} can be extended to  a $\K Q_0$-bilinear map $\K Q\to \K Q$ using the Leibniz rule, so it becomes a derivation of $\K Q$. Then, the map induces a  unique derivation of $A$ if and only if $\d(I) \subseteq I$.   
\end{rem}

\section{A left exact sequence of cohomology groups}\label{section:left-exact}

Let $C$ be a tilted algebra, $E=\Ext{C}{2}{DC}{C}$  and $B=C\ltimes E$ the corresponding cluster-tilted algebra. Upon applying the functor $\Hom{B^e}{B}{-}$ to the short exact sequence (\ref{eqn:short-BB}) of section \ref{subsec:cluster-tilted} we obtain a long exact sequence of cohomology groups
$$\SelectTips{eu}{10}\xymatrix@C=7pt{ 0\ar[r] & \Hom{B^e}{B}{E}\ar[r]&\Hom{B^e}{B}{B}\ar[r]&\Hom{B^e}{B}{C}\ar[r]&\Ext{B^e}{1}{B}{E}\ar[r]^{\iota} &\Ext{B^e}{1}{B}{B}\ar[r]^{\overline{p}} &\Ext{B^e}{1}{B}{C}\ar[r] &\cdots }$$
Our first task is to compare the cohomology groups of $C$, that is $\HH{i}{C}$, to those of $B$ with coefficients in $C$, that is $\Ho{i}{B}{C} = \Ext{B^e}{i}{B}{C}$.  The following lemma will be useful in the sequel.

\begin{lem}\label{lem:hom-E}
\begin{enumerate}[$a)$]\ 
	\item There is an isomorphism $\HH{0}{C} \simeq \Ho{0}{B}{C}$,
	\item For every $n\geqslant 1$, there is a  monomorphism $\HH{n}{C} \hookrightarrow \Ho{n}{B}{C}$.
\end{enumerate}
\end{lem}

\begin{proof} We use proposition \ref{prop:res-rad}. Let $\r_C$ and $\r_B$ be the radicals of $C$ and $B$, respectively. Because $B=C\ltimes E$ is a trivial extension, we have $\r_B =\r_C \oplus E$. In particular, the projection \mor{p}{B}{C} restricts to a natural retraction $\r_B \to\r_C$ which we still denote by $p$. Further, we let \mor{q}{\r_C}{\r_B} be its right inverse (section), and   $-^* = \Hom{C_0^e}{-}{C}$.

Note that we have $B_0 = C_0$. The groups $\Ext{C^e}{i}{C}{C}$ are the cohomology groups of the upper complex in the diagram below, whereas $\Ext{B^e}{i}{B}{C}$ are those of the  lower one. 
$$\SelectTips{eu}{10}\xymatrix{
0\ar[r] & C^{C_0} \ar[r]^-{d^1_C}\ar@{=}[d]& \Hom{C_0^e}{\r_C}{C}\ar[d]^{p^*}\ar[r]^-{d^2_C} & \Hom{C_0^e}{\r_C \otimes \r_C}{C}\ar[r] \ar[d]^{(p\otimes p)^*} &\cdots\\
0\ar[r] & C^{C_0} \ar[r]^-{d^1_B}& \Hom{C_0^e}{\r_B}{C}\ar[r]^-{d^2_B}&\Hom{C_0^e}{\r_B \otimes \r_B}{C}\ar[r]& \cdots
}$$

A direct computation shows that $\left(p^{\otimes \bullet}\right)^\ast$ defines a morphism of complexes which is in fact a section admitting $\left(q^{\otimes \bullet}\right)^\ast$ as retraction. This shows statement $b)$. For the isomorphism of statement $a)$ use the fact that  $p^\ast$ is injective and we have
$$\HH{0}{C} = \Ext{C^e}{0}{C}{C} = \ker{d^1_C} = \ker{p^* d^1_C} = \ker{d^1_B} = \Ext{B^e}{0}{B}{C} = \Ho{0}{B}{C}.$$  and we are done.

\end{proof}

Let us denote by $\ts{H}^n(p)$ and by $\ts{H}^n(q)$ the maps induced in cohomology by  $\left(p^{\otimes \bullet}\right)^\ast$ and  $\left(q^{\otimes \bullet}\right)^\ast$, respectively.  The next step is to extract a left exact  sequence from the long exact cohomology sequence involving only the degree one terms. For this sake, we define $\f$ to be the composition
$$\SelectTips{eu}{10}\xymatrix{\HH{1}{B} = \Ext{B^e}{1}{B}{B} \ar[r]^(.6){\overline{p}} & \Ext{B^e}{1}{B}{C} \ar[r]^(.40){\ts{H}^1(q)}& \Ext{C^e}{1}{C}{C} = \HH{1}{C} }.$$

\begin{rem} Note that \mor{\f}{\HH{1}{B}}{\HH{1}{C}} is the map that sends the class of a map $\d$ from $B$ to $B$ to that of  the map $p\d q$ from $C$ to $C$. A straightforward computation, as done in \cite{AR09}, shows that if $\d$ is a derivation (or an inner derivation), then so is $p\d q$. Thus, our map $\f$ is exactly the map $\HH{1}{B} \to \HH{1}{C}$ considered in \cite{AR09, ARS12}. 
 
\end{rem}

In order to obtain the desired 3-term sequence we need the following:

\begin{lem}\label{lem:new} $\ker{\f} = \ker{\ol{p}}$.
	
\end{lem}
\begin{proof}
	Clearly $\f = \ts{H}^1(q)\ol{p}$ implies immediately that $\ker{\ol{p}}\subseteq \ker{\f}$. Let thus $\d$ be a derivation whose class belongs to $\ker{\f}$. Thus $p\d q$ is an inner derivation of $C$, that is there exists $c\in C$ such that $p\d q =[c,-]$. Write  $\d_c = [c,-]$. Replacing $\d$ by $\d-\d_c$ we can assume that $p\d q=0$, that is $f=p \d$ equals zero when restricted to $C$. Now $f$ being a derivation on $B$ which is zero on $C$, is a morphism of $C\textrm{-}C$-bimodules \mor{f}{E}{C}.  Indeed, let $e \in E$ and $c \in C$ then $f(ec)=ef(c)+f(e)c=f(e)c$, and similarly $f(ce)=cf(e)$. Let now \mor{\c}{i}{j} be a new arrow (thus, a generator of $E$ as $C\textrm{-}C$-bimodule) then $f$ sends $\c\in e_iE e_j$ into $e_iC e_j$ which is zero, because $C$ is triangular. Therefore, $f=0$ on $E$, so $\ker{\f} = \ker{\ol{p}}$.
\end{proof}

\begin{cor}\label{cor:exact-left}
	There is an exact  sequence 
$$\SelectTips{eu}{10}\xymatrix{
0\ar[r] & \Ho{1}{B}{E} \ar[r]^\iota &\HH{1}{B} \ar[r]^\f & \HH{1}{C}
}$$
with $\f = \ts{H}^1(q)\, \ol{p}$.

\end{cor}
\begin{proof}
	Because $C$ is connected and triangular, its centre is isomorphic to $\K$. Lemma \ref{lem:hom-E} then gives $\Hom{B^e}{B}{C} \simeq \K$.  In addition, the map $\Hom{B^e}{B}{p}: \Hom{B^e}{B}{B} \to  \Hom{B^e}{B}{C}  $ appearing in the long exact sequence of cohomology groups is not zero, because the identity (on $C$) belongs to its image. Thus, this map is surjective, and hence $\iota = \Ext{B^e}{1}{B}{q}$ is injective. Using the notation of Lemma \ref{lem:new}, we have an  exact sequence
$$\SelectTips{eu}{10}\xymatrix{
0\ar[r] & \Ho{1}{B}{E} \ar[r]^\iota &\HH{1}{B} \ar[r]^{\overline{p}} & \Ho{1}{B}{C}
}$$
and a map \mor{\ts{H}^1(q)}{\Ho{1}{B}{C}}{\HH{1}{C}} such that $\f = \ts{H}^1(q)\, \ol{p}$. Invoking Lemma \ref{lem:new} completes the proof.
\end{proof}

\section{The surjectivity of $\f$}\label{section:surjectivity}

Our next step is to show that $\f$ is surjective (as  was shown in \cite{AR09,ARS12} under some additional hypotheses). Thus given a derivation $\d$ of $C$, we want to extend it to a derivation $\td$ of $B$. We  proceed in two steps:
\begin{enumerate}
	\item First of all, we consider $\d$ as a $\K Q_0$-bilinear derivation from $\K Q$ to itself that sends each relation $r\in \R$ to $I$. We extend $\d$ to  a map \mor{\td}{\K \tilde{Q}}{\K \tilde{Q}} by defining it on the new arrows.  Extending it by using the Leibniz rule we obtain a derivation of $\K \tilde{Q}$. We then show that $\td$  vanishes (up to cyclic permutation) on the potential $W$;
	\item Finally we show that in fact $\td (\t{I}) \subseteq \t{I}$ so that  in fact $\td$ is a derivation of $B$.
\end{enumerate}

\medskip

We recall that two paths $u,v$ in a quiver are said to be \emph{parallel} if $s(u)=s(v)$ and $t(u)=t(v)$, and \emph{antiparallel} if  $s(u)=t(v)$ and $t(u)=s(v)$.

We need  to define $\d$ on the new arrows, which, as already observed,  are in bijection with the elements of $\R$. Let $r_i\in \R $ be a relation from $x$ to $y$ in $\K Q$. Since $\d(r_i)\in e_x I e_y$, there exist scalars $b_{ik}$,  paths $u_{ik}, v_{ik}$ and relations $r_{j_k}$ such that
\begin{equation}\label{eqn:extend-d}
\d(r_i) = \sum_{k=1}^m b_{ik} u_{ik} r_{j_k} v_{ik}.	
\end{equation}

Since $r_i$ is parallel to $u_{ik} r_{j_k} v_{ik}$, and $r_i$ is anti-parallel to the corresponding new arrow $\c_i$, then each $v_{ik} \c_i u_{ik}$ is a path anti-parallel to $r_{j_k}$. 
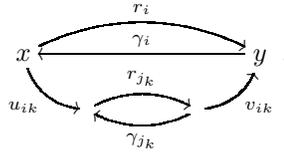
$$\SelectTips{eu}{10}\xymatrix@C=15pt@R=12pt{
x \ar@/_/[dr]_{u_{ik}}\ar@/^1pc/[rrrr]^{r_i} & 		&& & y\ar[llll]_{\c_i}\\
									&\ar@/^/[rr]^{r_{j_k}}	& &\ar@/^/[ll]^{\c_{j_k}}\ar@/_/[ru]_{v_{ik}} &}.$$
For a given arrow $\c_j$ we want to collect all the terms $v_{ik} r_i u_{ik}$ where $\c_j = \c_{j_k}$ with $r_{j_k}$ appearing in the expression of $\d(r_i)$ (equation  (\ref{eqn:extend-d}), above). More precisely, for a fixed new arrow $\c_j$ define
$$\mathcal{E}_j = \{(i,k)|\ r_{j_k} = r_j \mbox{ in the expression of } \d(r_i)\}$$

\begin{lem}\label{lem:delta-W}
Let \mor{\d}{\K Q}{\K Q} be a derivation such that $\d(I) \subseteq I$. Then the map \mor{\td}{\K \tilde{Q}}{\K \tilde{Q}} defined on the arrows of $\tilde{Q}$ by
$$\td(\a) = \left\{ \begin{array}{ll} \d(\a) 								& \mbox{ if } \a \mbox { is an old arrow};\\
						-\displaystyle  \sum_{(i,k)\in \mathcal{E}_j} b_{ik} v_{ik} \c_i u_{ik}	& \mbox{ if } \a = \c_j. \end{array}\right.$$
and extended by the Leibniz rule is a derivation of $\K \tilde{Q}$,  satisfying $\td(W)=0$ up to cyclic permutation.
\end{lem}
\begin{proof}
 The only thing we have to prove is that $\td(W)$ is zero up to cyclic permutation, but this follows readily from:
\begin{eqnarray*}
 \td(W) 	& = & \sum_{i=1}^n \td(r_i \c_i)\\
 		& = & \sum_{i=1}^n \d(r_i) \c_i + \sum_{j=1}^n r_j \td(\c_j)\\
 		& = & \sum_{i=1}^n \sum_{k=1}^m b_{ik} u_{ik} r_{j_k} v_{ik} \c_i+ \sum_{j=1}^n r_j\sum_{(i,k)\in \mathcal{E}_j} -b_{ik} v_{ik} \c_i u_{ik} \\
 		& = & \sum_{j=1}^n r_j\sum_{(i,k)\in \mathcal{E}_j} (b_{ik} u_{ik} r_j v_{ik} \c_i -b_{ik} r_j v_{ik} \c_i u_{ik})
\end{eqnarray*}
which is zero up to cyclic permutation.
\end{proof}

Let us now show that $\td(\tilde{I}) \subseteq \t{I}$. The ideal $\tilde{I}$ is generated by the partial derivatives $\partial_\a W$ of the potential  $W$ with respect to the arrows of $\tilde{Q}$. If $\c$ is a new arrow, then $\partial_\c(W)$ is an old relation (thus an element of $I$), and then $\td(\partial_\c W) = \d(\partial_\c W) \in I \subseteq \tilde{I}$, since $\d$ is a derivation of $C$.  Thus we only need to look at the partial derivatives with respect to the old arrows, or, equivalently, we only need to look at the new relations. Before proving the desired result we need some preliminary observations.

For any vertex $i$ in $Q$ define:
\medskip

$\begin{array}{rcl}
 W_1 &=& \mbox{sum of all terms in } W \mbox{ that pass through } i\\
	  &=&\displaystyle\sum_{\c : s(\c)=i} \c\ \partial_\c W;\\
 	& & \\
 W_2 &=& \mbox{sum of all terms in } W \mbox{ that do not pass through } i \mbox{ and contain an arrow}\\
 	& & \beta \mbox{ such that }\td(\beta) \mbox{ passes through } i;\\
 \\
 W_3&=&\textup{sum of all other terms in $W$}.
\end{array}$

\medskip

Since the quiver $Q$ has no oriented cycles, we can number its vertices in such a way that for every arrow $\a \in Q_1$ we have $s(\a) < t(\a)$. We fix such a numbering in the sequel.

\begin{lem}
 Let $\Theta$ be the set of all the arrows $\b$ appearing in $W_2$ such that $\td(\b)$ passes through $i$. Then $\displaystyle W_2 =\sum_{\beta \in \Theta}\beta \partial_{\beta}\!W$.
\end{lem}
\begin{proof}
We show that a cycle $w$ appearing in $W_2$ contains exactly one arrow $\b$ such that $\td(\b)$ passes through the vertex $i$. Assume the contrary, that is, there are two such arrows, say $\b$ and $\beta'$. By \cite[2.1]{AR09} one of them must be an old arrow, say $\b$.  It then follows from the definition of $\td$ that $\td(\b) = \d(\b)$ is a linear combination of paths  containing only old arrows, which go from $s(\b)$ to $t(\b)$, and at least one of them passes through $i$, so that $s(\b) < i < t(\b)$. We now show that the cycle $w$ can contain at most one old arrow having this property. Assume to the contrary that, up to cyclic permutation, $w = \b u \beta' u'$ with $u, u'$ paths in  $Q$, and $\b, \beta'$ two arrows  such that $\td(\b) = b_1 b_2+ $ other terms, $\td(\beta')= b'_1 b'_2$ + other terms, with $b_1, b'_1$ paths ending at $i$.
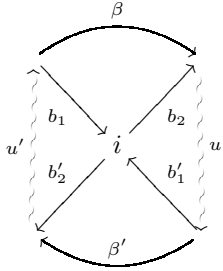
$$\SelectTips{eu}{10}\xymatrix{
\ar@/^1pc/[rr]^\b \ar[dr]_{b_1}		&					& \ar@{~>}[dd]^u\\
						&i  \ar[ur]_{b_2}\ar[dl]_{b'_2}	&\\
\ar@{~>}[uu]^{u'}& &\ar@/^1pc/[ll]_{\beta'}\ar[ul]_{b'_1}}$$

Suppose first that $\b'$ is an old arrow. Then $b_1'$ and $b_2'$ are paths in $Q$.  Since $w= \b u \beta' u'$ is a summand of the potential, it contains exactly one new arrow. If this new arrow occurs in $u'$, then $b_2 u b'_1$ is a cycle in $\t{Q}$ consisting of old arrows, a contradiction. If the new arrow appears in $u$, a similar argument works. Thus $\b'$ is a new arrow and we have by definition,  $$\tilde\delta(\beta')= -\sum_j b_jv_j\c_ju_j,$$
where $\c_j$ are new arrows, $v_j,u_j$ are old paths and $v_j\c_ju_j $ is a path from $s(\beta') $ to $t(\beta')$. But
$$ t(\beta')\leqslant s(\beta)<i<t(\beta)\leqslant s(\beta')$$
so $v_j\c_ju_j$ cannot pass through $i$, a contradiction.

We can thus write $W_2 =\displaystyle  \sum_{\b \in  \Theta} \b w_{\b}$, where each $w_\b$ is a linear combination of paths.  By definition of $W_2$, $w_{\b}$ cannot pass through $i$, hence all the terms of $W$ which contain an arrow $\b\in \Theta$ appear exactly once in this sum. Therefore $w_{\b} = \partial_\b W$.
\end{proof}

We are now able to show that for every old arrow $\a$ we have $\td\left(\partial_\a W\right) \in \t{I}$. We do so in two steps. 

For a fixed vertex $i\in Q_0$ let $\t{I}_{\ne i}$ be the ideal of $\K \t{Q}$ generated by all the partial derivatives $\partial_\e W$ such that $\e$ does not end at $i$. 

\begin{lem}\label{lem:simp-gamma}
	If $\sum_{\c:s(\c)=i} \c w_\c$ is a linear combination of cycles which belongs to $\t{I}_{\ne i}$, then $w_\c \in \t{I}_{\ne i}$ for every $\c$ such that $s(\c) = i$.
\end{lem}
\begin{proof}
Write 	$\sum_{\c:s(\c)=i} \c w_\c = \sum_j b_j u_j r_j v_j$ with $b_j$ scalars, $r_j$ generators of $\t{I}_ {\ne i}$ and  $u_j$, $v_j$ paths. Note that by definition the first arrow of the $u_j$ is one of the arrows $\c$ starting at $i$, thus $u_j = \c_j u'_j$. Then
$$\sum_{\c:s(\c)=i} \c w_\c = \sum_j b_j \c_ j u'_j r_j v_j$$
and, upon comparing the terms, we obtain
$$w_\c   = \sum_{j:\c=\c_j} b_j  u'_j r_j v_j \in \t{I}_{\ne i}.$$
\end{proof}

We can now prove  the required statement.
\begin{lem} Let $\a$ be an old arrow, and $\td$ defined as before. Then $\td(\partial_\a W)\in \t{I}$.
\end{lem}
\begin{proof}
 We showed in \ref{lem:delta-W}  that up to cyclic permutation $\td(W) = 0$.  Take any $\a\in Q_1$ and let $i=s(\a)$. Then by construction we have $W = W_1 + W_2+W_3$. Thus using the Leibniz rule, we obtain (up to cyclic permutations):
 \begin{equation}\label{eqn:eq1}
  0  =\!\! \sum_{\c:s(\c)=i}\!\!\! \td(\c) \partial_\c W +\!\!\!  \sum_{\c:s(\c)=i}\!\!\! \c \td (\partial_\c W) + \sum_{\b \in \Theta} \td(\b) \partial_\b W + \sum_{\b \in \Theta} \b \td (\partial_\b W)   + \td\left(W_3\right).
 \end{equation}
By definition, the terms $\partial_\c W$ and $\partial_\b W$ appearing in the expression above  belong to $\tilde{I}_{\ne i}$. Moreover, the terms in the first two sums involve  paths passing through $i$, and the last term, as well as the fourth sum  consist of paths not passing through $i$. 
Concerning  the third sum, each $\td(\b)$ is a linear combination of paths, some of them (at least one) passing through $i$, and some (may be none) not passing through $i$. Accordingly, for each $\b$ collect the terms passing through $i$ together, and call the result $\td_i(\b)$. Collect the remaining terms to obtain $\td_0(\b)$. Thus the third sum above is
$$\sum_{\b \in \Theta} \left(\td_i(\b) + \td_0(\b) \right) \partial_\b W.$$
Altogether $\td(W)$ splits into the sum of terms passing through $i$ and the sum of terms not passing through $i$; and both sums are equal to zero. Therefore we have the following equation of cyclic words.
  \begin{equation}\label{eqn:213}
0=  \sum_{\c : s(\c)=i} \td(\c)\ \partial_\c \! W
 +\sum_{\c : s(\c)=i} \c\ \td(\partial_\c\!  W)
   +\sum_{\b\in \Theta }  \td_i(\b) \ \partial_{\b}\!  W.
 \end{equation}
Furthermore, we can view these cyclic words as paths starting at the vertex $i$. Then, the sum of these paths equals $0$.  Since each term $\partial_\c W$ and $\partial_\b W$ belongs to $\t{I}_{\ne i}$, so does $\displaystyle \sum_{\c:s(\c)=i} \c\td\left( \partial_\c W\right)$. By Lemma \ref{lem:simp-gamma} this implies that $\td(\partial_\c W)\in \t{I}$ for each $\c$ with $s(\c)=i$. So, $\td(\partial_\a W) \in \t{I}.$ 
\end{proof}

The preceding lemmata show that $\f(\tilde\delta)=\delta$. Thus $\f$ is surjective, completing the proof of the main theorem:

\begin{thm}\label{thm:main} With the notations of section \ref{section:left-exact}  there is an exact sequence 
$$\SelectTips{eu}{10}\xymatrix{
0\ar[r] & \Ho{1}{B}{E} \ar[r]^\iota &\HH{1}{B} \ar[r]^\f & \HH{1}{C}\ar[r]& 0 
}.$$
	
\end{thm}
\qed	

\medskip

We have an immediate consequence of this theorem. We recall first that if $B$ is the relation extension of $C$, then $C$ is not uniquely determined by $B$ (see, for instance, \cite{ABS08a}).

\begin{cor}
	Let $C, C'$ be tilted algebras and $E=\Ext{C}{2}{DC}{C}, E'=\Ext{C'}{2}{DC'}{C'}$. If $B=\C\ltimes E \simeq \C'\ltimes E'$, then $\HH{1}{C} \simeq \HH{1}{C'}$ and $\Ho{1}{B}{E} \simeq \Ho{1}{B}{E'}$.
\end{cor}
\begin{proof}
	Under the stated hypothesis, $C$ and $C'$ are tilted algebras of the same type (see \cite{ABS08a}). Because of \cite[4.2]{Hap89} we have $\HH{1}{C} \simeq \HH{1}{C'}$. The theorem then implies immediately that we also have  $\Ho{1}{B}{E} \simeq \Ho{1}{B}{E'}$.
\end{proof}

\section{Interpretation of the Kernel}\label{section:kernel}
In this section we proceed to  relate our result to those of  \cite{AR09, ARS12}. There, under some hypotheses on $B$ or $C$, the kernel of the map \mor{\f}{\HH{1}{B}}{\HH{1}{C}} was computed by means of an equivalence relation on the set of new arrows which we now describe.

Given a strongly minimal  relation $\sum_{i=1}^m a_i w_i$ in $\t{I}$, either it is a relation of $I$, or there exist exactly $m$ new arrows $\c_1,\ldots,\c_m$ and old paths $u_i, v_i$ such that $w_i = u_i \c_i v_i$ (see \cite[3.1]{AR09} or \cite[3.1]{ARS12}). We let  $\approx$ be the smallest equivalence relation on  the set of new arrows  such that $\c \approx \c'$ whenever $\c$ and $\c'$ are two new arrows appearing in a strongly minimal relation. Finally, we let  $n=n_{B,C}$ be the number of equivalence classes for $\approx$. 

\begin{cor} Let $B= C\ltimes E$ be such that $B$ is tame or $C$ is constricted. Then we have $\Ho{1}{B}{E} \simeq \K^n$.
 \end{cor}

\begin{proof} This follows immediately from Theorem \ref{thm:main} and the main result of \cite{ARS12}.
\end{proof}

We now prove that $\End{C^e}{E}$ is always a subspace of $\Ho{1}{B}{E}$. We recall that $\Ho{1}{B}{E}$ is $\ker{d^2} / \im{d^1}$ in the complex below
$$\SelectTips{eu}{10}\xymatrix{0\ar[r] & E^{B_0} \ar[r]^-{d^1}& \Hom{B_0^e}{\r_B}{E}\ar[r]^-{d^2} &\Hom{B_0^e}{\r_B \otimes \r_B}{E}\ar[r]  &\cdots }$$

\begin{lem}\label{lem:End-sub-HH}
 We have inclusions of vector spaces
$$\End{C^e}{E} \subseteq \Ho{1}{B}{E} \subseteq \End{C^e}{E} \oplus \frac{\Hom{C_0^e}{\r_C}{E}}{\im{d^1}}$$ 
\end{lem}

\begin{proof}
Because $B=C\ltimes E$, we have $\r_B= \r_C \oplus E$. Also, $B_0 = C_0$ so that
$$\Hom{B_0^e}{\r_B}{E} = \Hom{C_0^e}{\r_C}{E} \oplus \End{C_0^e}{E}.$$
Thus, any $f\in \Hom{B_0^e}{\r_B}{E}$ can be written uniquely as $f= f_1 + f_2$ with $f_1 \in \Hom{C_0^e}{\r_C}{E}$ and $f_2 \in  \End{C_0^e}{E}$. Let $b=c+e \in \r_B$, where $c\in \r_C$ and $e\in E$.  Then 
$$f(b) = f_1(c) + f_2(e).$$
The statement of the lemma will follow easily from the next three claims.

\begin{enumerate}
 \item We first claim that $\im{d^1} \subseteq \Hom{C_0^e}{\r_C}{E}$. Indeed, let $f\in \im{d^1}$. There exists $e_0\in E^{B_0}$ such that $f=[e_0,-]$. But then, for every $b = c+e \in \r_B$ (with $c\in \r_C$ and $e\in E$) we have $f(b)= [e_0, c] + [e_0,e]$. Now $[e_0,e]= e_0e-ee_0 = 0$ because $E^2 =0$. Therefore $f(b) = [e_0,c] = f_1(c)$, that is, $f=f_1$.

\item Next, we see that $ \End{C^e}{E}\subseteq \ker{d^2}$. Let indeed $g\in\End{C^e}{E }{×}$. In particular $g\in \End{C_0^e}{E}$ so $g$ induces $f\in \Hom{B_0^e}{\r_B}{E}$ by $f(c+e) = g(e)$ (for $c\in \r_C$ and $e\in E$). We want to prove that $f$ is a derivation. Let $b,b'\in \r_B$ be such that $b=c+e,\ b'=c'+e'$ (with $c,c'\in \r_C$ and $e,e'\in E$). Then

$$f(bb') = f\left((c+e)(c'+e')\right) = g(ec'+ce') = g(e)c' + cg(e') = f(b)c' +c f(b') = f(b)b'+bf(b')$$ because $f(b)e' = ef(b')=0$.

\item Finally, we prove that $\ker{d^2} \subseteq \Hom{C_0^e}{\r_C}{E} \oplus \End{C^e}{E}$. Write $f\in \ker{d^2}$ as before in the form $f=f_1+f_2$ with $f_1 \in \Hom{C_0^e}{\r_C}{E}$ and $f_2 \in  \End{C_0^e}{E}$. We claim that in fact $f_2$ is a morphism of $C\textrm{-}C$-bimodules. Let $c\in C$ and $e\in E$, then $ce\in E$ so that $f_2(ce)= f(ce)$. Now, $f$ is a derivation, hence
$$f_2(ce) = f(ce) = f(c)e + cf(e) = f_1(c)e+cf_2(e) = cf_2(e) $$ because $f_1(c)e\in E^2=0$. Similarly, $f_2(ec) = f_2(e)c$. This proves that $\ker{d^2} \subseteq \Hom{C_0^e}{\r_C}{E}  + \End{C^e}{E}$. Because $\End{C^e}{E}\subseteq \End{C_0^e}{E}$, the sum is direct.
\end{enumerate}

\end{proof}

We now see that, in general, $\Ho{1}{B}{E}$ depends on the direct decompositions of $E$ as $C\textrm{-}C$-bimodule. We recall first that, as 
$C\textrm{-}C$-bimodule, $E$ is generated by the new arrows. If two new arrows occur in a strongly minimal relation, this means that they are somehow yoked together in $E$. Direct decompositions of $E$ as $C\textrm{-}C$-bimodule are studied in  \cite{ABDLS12} from which we import the following result. We include a proof for the benefit of the reader.
\begin{lem} As $C\textrm{-}C$-bimodule, $E$ decomposes as the direct sum of $n$ nonzero summands.
 \end{lem}
\begin{proof}
 Let $\cal{S}_1,\ldots \cal{S}_n$ be the equivalence classes of new arrows, with  $\cal{S}_j = \{\c_{j1},\ldots, \c_{js_j}\}$. Further, let $E_j$ be the $C\textrm{-}C$-bimodule generated by the arrows of $\cal{S}_j$, so its elements are of the form $x_j = \sum_{k=1}^{s_j} b_{jk} u_{jk} \c_{jk} v_{jk}$ where, as before, the $b_{jk}$ are scalars while $u_{jk}$ and $v_{jk}$ are (classes of) paths. We thus have a natural epimorphism \mor{\eta}{\bigoplus_{j=1}^n E_j}{E} given by $\eta(x_1,\ldots, x_n) = \sum_{j=1}^n x_j$. We show that it is also a monomorphism. Let $(x_1,\ldots x_n)$ be a non-zero element of the kernel of $\eta$ with the additional property that the number of non-zero elements  among the $x_j$ is minimal. Thus, we have a relation
$$\sum_{j=1}^n \sum_{k=1}^{s_j}  b_{jk} u_{jk} \c_{jk} v_{jk} = 0.$$
If $n=1$ there is nothing to show, so assume there are at least two classes occurring in the relation (that is two values of $j$). By definition, this relation is not strongly minimal, so there must be a strongly minimal relation involving the same paths
$$\sum_{j=1}^n \sum_{k=1}^{s_j}  b'_{jk} u_{jk} \c_{jk} v_{jk} = 0.$$
By definition of the equivalence relation $\approx$ only one class appears in the second relation. By subtracting a multiple of this second relation from the first one, we can reduce the number of non-zero terms in the original one, still getting an element of $\ker{\eta}$, a contradiction.
\end{proof}

Assume $E=E_1 \oplus \cdots \oplus E_n$, with the $E_i$ nonzero. Then the identity on each $E_i$ induces clearly an endomorphism of $E$ as $C\textrm{-}C$-bimodule. Since these endomorphisms are linearly independent, we get the following corollary.

\begin{cor} We have $\ts{dim}_\K \Ho{1}{B}{E} \geqslant n$, and if equality holds,then  $$\Hom{C^e}{E_i}{E_j} = \left\{ \begin{array}{cl} \K & \mbox{if } i=j, \ \\ 0 & \mbox{otherwise.} \end{array} \right.$$
 
\end{cor}
\begin{proof}
 Write $E=\bigoplus_{j=1}^nE_j$ so that, as vector spaces, we have $\End{C^e}{E} = \bigoplus_{i,j=1}^n \Hom{C^e}{E_i}{E_j}$. The identity maps on each $E_j$ provide $n$ linearly independent elements in $\End{C^e}{E}$ which, by Lemma \ref{lem:End-sub-HH} , is contained in $\Ho{1}{B}{E}$. This proves the first statement. If in addition equality holds, we must have $n=\ts{dim}_\K \End{C^e}{E}$, and the conclusion follows.
\end{proof}

We know that equality occurs under the hypothesis of \cite{ARS12}, so we get the following corollary.

\begin{cor}
 Assume $B = C\ltimes E$ is such that $B$ is tame or $C$ is constricted, then the indecomposable summands of $E$ as $C\textrm{-}C$-bimodule are pairwise orthogonal bricks.
\end{cor}\qed

Actually, we state:

\begin{conj} Let $C$ be a tilted algebra, and $E=\Ext{C}{2}{DC}{C}$. Then the indecomposable summands of $E$ as $ C\textrm{-}C$-bimodule are pairwise orthogonal bricks.
\end{conj}

\section{An example} Let $\K$ be a field of characteristic different from $2$. Consider the tilted algebra $C$ given by the bound quiver of figure \ref{fig:ex-C}.
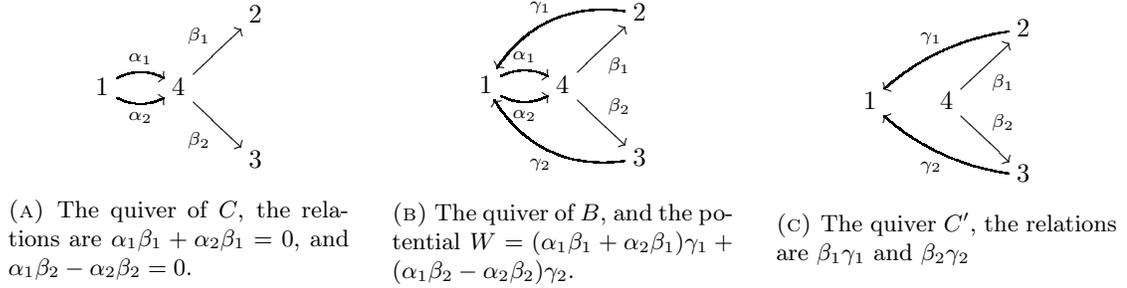
\begin{figure}[h!]\label{fig:example}
	\begin{subfigure}[h]{0.3\textwidth}
\begin{center}$\SelectTips{eu}{10}\xymatrix@R=15pt@C=18pt{ 
		 &				&2\\
1\ar@/^/[r]^{\a_1} \ar@/_/[r]_{\a_2}& 4 \ar[ur]^{\b_1} \ar[dr]_{\b_2}& \\
		&			& 3}$\end{center}
	\caption{The quiver of $C$, the relations are  $\a_1 \b_1 + \a_2 \b_1 =0$, and $\a_1\b_2 - \a_2\b_2 =0$.}
	\label{fig:ex-C}
	\end{subfigure}\hspace{0.03\textwidth}
	\begin{subfigure}[h]{0.3\textwidth}
\begin{center} $\SelectTips{eu}{10}\xymatrix@R=15pt@C=18pt{ 
		 &				&2\ar@/_1pc/[lld]_{\c_1} \\
1\ar@/^/[r]^{\a_1} \ar@/_/[r]_{\a_2}& 4 \ar[ur]_{\b_1} \ar[dr]^{\b_2}& \\
		&			& 3\ar@/^1pc/[llu]^{\c_2}}$\end{center}
	\caption{The quiver of $B$, and the  potential $W =(\a_1 \b_1 + \a_2 \b_1)\c_1 +(\a_1\b_2 - \a_2\b_2)\c_2$.  }	
	\label{fig:ex-B}
	\end{subfigure}\hspace{0.03\textwidth}
	\begin{subfigure}[h]{0.3\textwidth}
\begin{center}$\SelectTips{eu}{10}\xymatrix@R=15pt@C=18pt{ 
		 &				&2\ar@/_/[lld]_{\c_1}\\
1& 4 \ar[ur]_{\b_1} \ar[dr]^{\b_2}& \\
		&			& 3\ar@/^/[llu]^{\c_2}}$ \end{center}
	\caption{The quiver $C'$, the relations are $\b_1 \c_1$ and $\b_2 \c_2$}
	\label{fig:ex-c'}\end{subfigure}
\caption{The cluster-tilted algebra $B$ as relation extension of two tilted algebras $C$ and $C'$.}
\end{figure}
The relations for $C$ form a strongly minimal set of relations. The corresponding cluster-tilted algebra is $B$, given by the quiver of figure \ref{fig:ex-B} with potential  $W =(\a_1 \b_1 + \a_2 \b_1)\c_1 +(\a_1\b_2 - \a_2\b_2)\c_2$. In this case, the partial derivatives with respect to the old arrows are
\begin{eqnarray*}
	\partial_{\a_1} W = \b_1 \c_1 + \b_2 \c_2, & \mbox{\ \ \ \ } & \partial_{\a_2} W = \b_1 \c_1 - \b_2 \c_2,\\
	\partial_{\b_1} W = \c_1 \a_1 + \c_1\a_2, & \mbox{\ \ \ \ } & \partial_{\b_2} W = \c_2 \a_1 - \c_2 \a_2.
\end{eqnarray*}
However, these relations do not form a system of strongly minimal relations, since  $\partial_{\a_1} W + \partial_{\a_2} W = 2 \b_1 \c_1$, and this leads to a monomial relation $\b_1 \c_1$. In a similar way we obtain a monomial relation $\b_2\c_2$. This shows that $\c_1 \not\approx \c_2$, hence there are two equivalence classes. Here the extension bimodule $E$ decomposes as $E_1 \oplus E_2$, with $E_1 = C \c_1 C$ and $E_2 = C\c_2 C$. Assume $f\in \Hom{C^e}{E_1}{E_2}$. Since $f(\c_1) = f(e_2 \c_1 e_1) = e_2 f(\c_1) e_1$ and this lies in $E_2$, we have $f=0$. Using similar arguments one gets that  $E_1$ and $E_2$ are orthogonal bricks.

On the other hand, let $C'$ be the algebra given by the bound quiver of figure \ref{fig:ex-c'}. Then $B$ is the relation extension $C'\ltimes E'$. In this case the two new arrows are $\a_1$ and $\a_2$, so that $E' = E'_1 \oplus E'_2$ with $E'_1 = C' \a_1 C'$ and $E'_2 = C'\a_2 C'$. As vector spaces, we have
\begin{eqnarray*}
	E'_1 = \langle \a_1, \a_1 \b_2, \c_2 \a_1, \c_2 \a_1 \b_2 \rangle & \mbox{\ and\ } &E'_2 = \langle \a_2, \a_1 \b_1, \c_1 \a_2, \c_1 \a_2 \b_1 \rangle.
\end{eqnarray*}
Assume $f\in \Hom{C^e}{E'_1}{E'_2}$. Since $f(\a_1) = f(e_1 \c_1 e_4) = e_1 f(\a_1) e_4$ lies in $E'_2$, there exists a scalar $\l$ such that $f(\a_1) = \l\a_2$. Now, we have
$$0 = f(\a_1 \b_1 ) =  f(\a_1) \b_1 = \l \a_2 \b_2$$
and this forces $\l=0$ so that  $\Hom{C^e}{E'_1}{E'_2} = 0$. Using similar arguments one can get that $E'_1$ and $E'_2$ are orthogonal bricks.

\section*{Acknowledgements}
The first author gratefully acknowledges partial support from NSERC of Canada, FQRNT of Qu\'ebec and the Universit\'e de Sherbrooke. The third author is supported by a grant from the NSA. The fourth author gratefully acknowledges partial support from the NSF Grant  DMS-1001637 and  the University of Connecticut.

\bibliographystyle{acm}
\bibliography{../../ReferenciaMat/biblio}

\end{document}